\documentclass[12pt,leqno]{amsart}
\usepackage{amsmath,amssymb,amsfonts}
\usepackage{eucal,enumerate}
\usepackage{xypic,graphicx,psfrag}
\usepackage{mathrsfs}
\usepackage{hyperref}

\usepackage{color} 

\usepackage{tikz}
\usetikzlibrary{decorations.pathreplacing}
\usepackage{scalefnt}

\newcommand \comment[1]{\textbf{[#1]}}		
\renewcommand \comment[1]{}		
\newcommand \mylabel[1]{\label{#1}}
\renewcommand \mylabel[1]{\label{#1}\comment{{\rm \{#1\} }}}

\newtheorem{lem}{Lemma}[section]

\newtheorem{thm}[lem]{Theorem}

\allowdisplaybreaks

\newcommand\cB{\mathcal{B}}
\newcommand\cBo{{\cB^\circ}}
\newcommand\cM{\mathscr{M}}
\newcommand\bbR{\mathbb{R}}
\newcommand\bbZ{\mathbb{Z}}

\newcommand\pP{\mathbb P}	
\newcommand\pQ{\mathbb Q}	

\newcommand\M{\mathbf{M}}
\newcommand\Kot{Kot\v{e}\v{s}ovec}

\allowdisplaybreaks

\begin{document}


\title{A $q$-Queens Problem.  VII.  \\
Combinatorial Types\\ of Nonattacking Chess Riders} 

\author{Christopher R.\ H.\ Hanusa}
\address{Department of Mathematics \\ Queens College (CUNY) \\ 65-30 Kissena Blvd. \\ Queens, NY 11367-1597, U.S.A.}
\email{\tt chanusa@qc.cuny.edu}

\author{Thomas Zaslavsky}
\address{Department of Mathematical Sciences\\ Binghamton University (SUNY)\\ Binghamton, NY 13902-6000, U.S.A.}
\email{\tt zaslav@math.binghamton.edu}

\begin{abstract}
On a convex polygonal chessboard, the number of combinatorial types of nonattacking configuration of three identical chess riders with $r$ moves, such as queens, bishops, or nightriders, equals $r(r^2+3r-1)/3$, as conjectured by Chaiken, Hanusa, and Zaslavsky (2019).  Similarly, for any number of identical 3-move riders the number of combinatorial types is independent of the actual moves.  
\end{abstract}

\subjclass[2010]{Primary 05A15; Secondary 00A08, 52C35.}

\keywords{Nonattacking chess pieces, fairy chess pieces, arrangement of lines}

\thanks{The first author gratefully acknowledges support from PSC-CUNY Research Award 61049-0049.}

\maketitle
\pagestyle{myheadings}
\markright{\textsc{A $q$-Queens Problem. VII. Combinatorial Types}}\markleft{\textsc{Hanusa and Zaslavsky}}

\section{Combinatorial Types}

Consider a chessboard, say an $n\times n$ square board, and a chess piece $\pP$ resembling the queen, bishop, and rook in that it has a fixed set of lines along which it can move and it can move any distance in either direction along those move lines.  Such pieces are known as \emph{riders} in fairy chess;\footnote{Chess with varied pieces, rules, or boards.} an example is the nightrider, which moves any distance in the directions of a knight's move.  

Now place several (labelled) copies of $\pP$ on the board in a nonattacking configuration, i.e., no piece lies on a move line of another piece.  
The piece $\pP_i$ in a particular location on the board divides the board into open regions by its move lines, each region determined by the two move lines of $\pP_i$ that bound it.  The other pieces must be inside some of these regions, as they cannot be on the move lines.  
Two nonattacking configurations with the same number of labelled pieces are said to have the same \emph{labelled combinatorial type} if for each pair $\pP_i$ and $\pP_j$, $\pP_j$ lies in the same region of the board with respect to $\pP_i$ in both configurations.  For example, there are six labelled (and three unlabelled) combinatorial types of nonattacking configuration for two copies of a piece with three move lines, shown in Figure~\ref{F:23}.

\medskip
\begin{figure}[htb]
	\includegraphics[scale=.45]{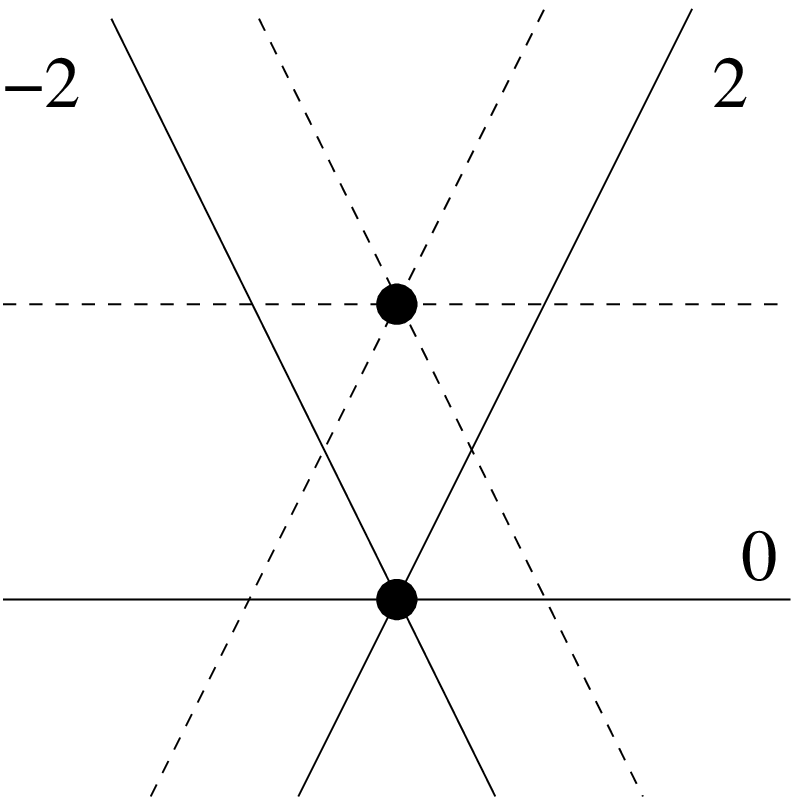} \qquad
	\includegraphics[scale=.45]{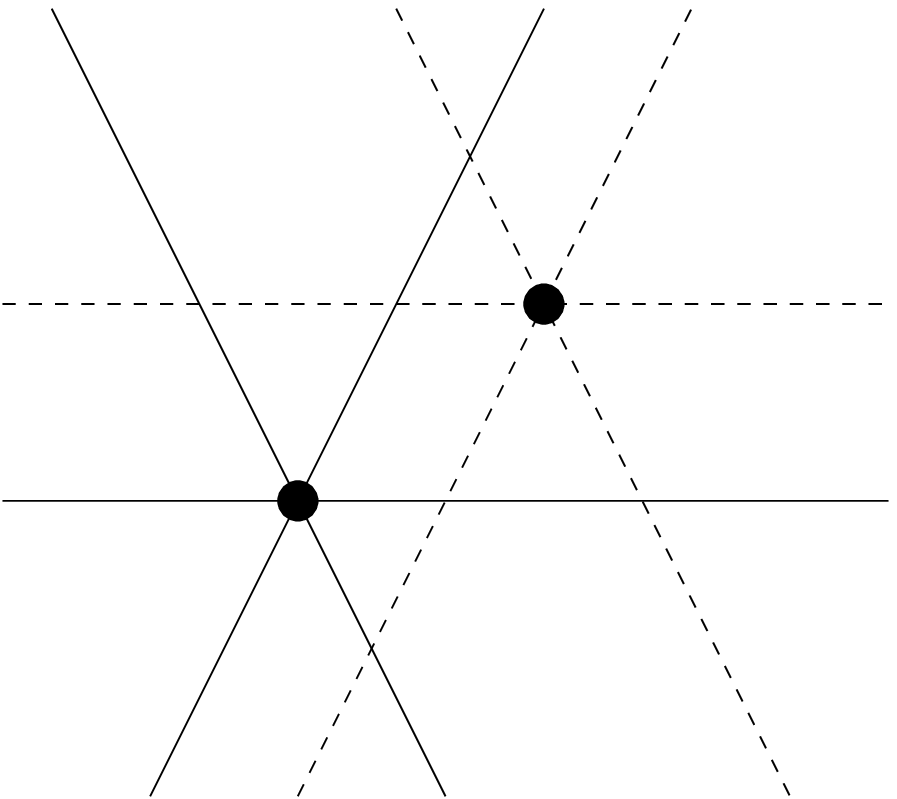} \qquad
	\includegraphics[scale=.45]{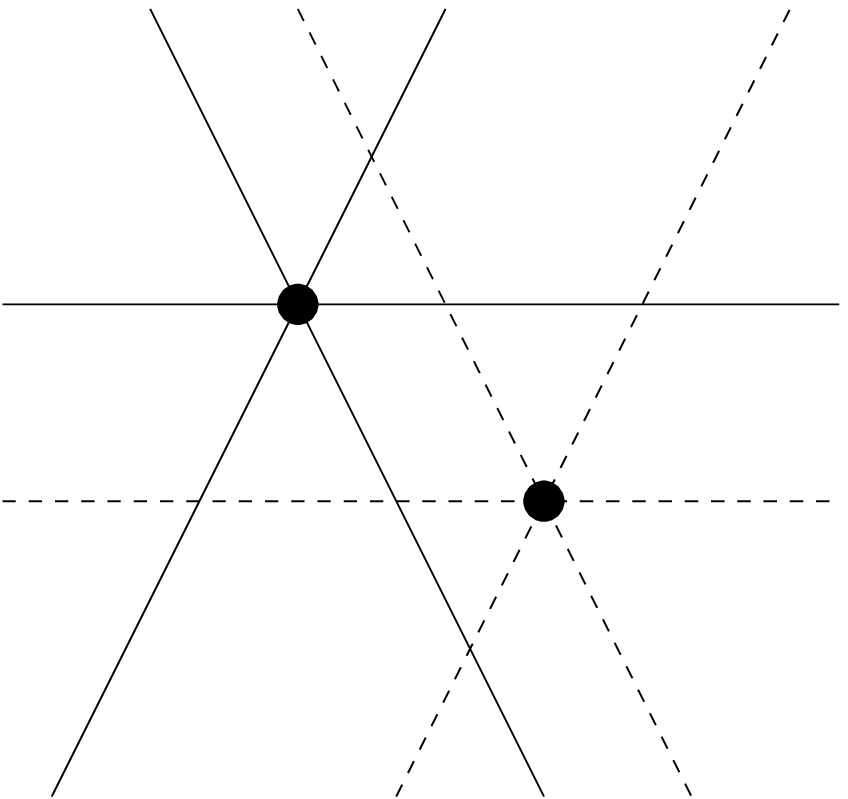} \qquad
	\caption{The three unlabelled combinatorial types of two (unlabelled) identical nonattacking riders with three moves along lines of slope $0$ and $\pm2$.  Since there are two ways to label the pieces in each type, there are six labelled combinatorial types.
	}
	\mylabel{F:23}
\end{figure}

How many combinatorial types of nonattacking configuration are there?  Call the number of unlabelled types $t_\pP(q)$.  \emph{A priori}, the answer could depend on the move lines of $\pP$, on the number $q$ of pieces, and on the board.  Happily, it turns out that the board itself does not matter because every possible combinatorial type can be realized on any sufficiently large board.  The set of move lines and the value of $q$ remain as relevant variables, which still are relatively complicated information.

Let us review the known data (Table \ref{Tb:qr}).  The number of combinatorial types is known for very few pieces and pieces with very few move lines.  Suppose there are $r$ move lines.  It is easy to see that $t_\pP(1) = 1$ and $t_\pP(2) = r$ \cite[Theorem 5.6]{QQs1}.  One might expect $t_\pP(q)$ to depend on the exact move lines for larger numbers of pieces, but Chaiken and we conjectured in \cite[Conjecture 4.4]{QQs3} that for any rider piece with $r$ move lines, $t_\pP(3) = r(r^2+3r-1)/3$.  Here we prove that conjecture.  We also show that when pieces have only three move lines the value of $t_\pP(q)$ does not depend on the exact lines, no matter how many pieces there are.  
\begin{table}
	\begin{tabular}{c|cccccc}
$q \backslash r$ & 1 & 2 & 3 & 4 & 5 & 6 \\ \hline
1     & 1 & 1 & 1 & 1 & 1 & 1 \\
2     & 1 & 2 & 3 & 4 & 5 & 6 \\
3     & 1 & 6 & 17 & 36 & 65 & 106 \\
4     & 1 & 24 & 151$^*$ & ${574_\pQ}^*$ & ? & ? \\
5     & 1 & 120 & 1899$^*$ & ${14206_\pQ}^*$ & ? & ? \\
6     & 1 & 720 & 31709$^*$ & ${501552_\pQ}^*$ & ? & ? \\
\end{tabular}
\medskip
\caption{The number of unlabelled combinatorial types of nonattacking configuration of $q$ identical riders with $r$ moves.  $^*$ indicates a number that was computed from \Kot's empirical formulas.  The subscript $_\pQ$ indicates a number that has been computed for queens but may depend on the piece.  The numbers marked by ?\ are unknown and may depend on the piece. }
\mylabel{Tb:qr}
\end{table}

On the other hand, $t_\pP(q)$ could depend on the set of move lines when $q,r\geq4$.  But data is hard to come by.  It is hard to compute values by direct counting when $q$ or $r$ is not tiny, i.e., $\leq2$.  For $q\geq3$ or $r\geq3$ we get $t_\pP(q)$ from a general theorem.  
A \emph{quasipolynomial} function $f(n)$ is a function given by a cyclically repeating sequence of polynomials.

\begin{lem}[{\cite[Theorems 4.1 and 5.3]{QQs1}}]\mylabel{L:-1}
The number $o_\pP(q;n)$, or $u_\pP(q;n)$, of nonattacking configurations of $q$ labelled, respectively unlabelled, copies of $\pP$ on an $n\times n$ square board is given by a quasipolynomial function of $n$ of degree $2q$.  

The number of combinatorial types is equal to $o_\pP(q;-1)$ or $u_\pP(q;-1)$, respectively.
\end{lem}

(The functions are related by $u_\pP(q;n) = o_\pP(q;n)/q!$.  The former is probably more interesting but solutions are obtained by means of labelled counts.)

Applying Lemma \ref{L:-1} requires knowing the quasipolynomial formula for the counting function.  That is hard, and few such formulas are known.  We rely on \Kot's heuristic results, especially \cite{ChMath}, for most of them.  
(We have found that, when they can be checked rigorously, as in \cite{QQs3} and its related papers, Kot\v{e}\v{s}ovec's formulas are always correct.)  
Table \ref{Tb:qr} shows the results.

\section{Background}\label{sec:back}

We need precise definitions.  The \emph{board} $\cB$ is a closed, convex polygonal region in the plane. 
(In \cite{QQs1, QQs3} the vertices are assumed rational; here that is unnecessary.)  
The working board of order $n$, on which we place pieces, is $(n+1)\cBo \cap \bbZ^2$, $\cBo$ denoting the interior of $\cB$.  
For example, take the square board $\cB = [0,1]^2$; then $(n+1)\cBo \cap \bbZ^2 = [n]^2$ (where $[n]$ means $\{1,2,\ldots,n\}$), which is indeed the ordinary $n\times n$ square chessboard.  

The \emph{move set} of a piece $\pP$ is the set $\M = \{m_1,\dots,m_r\}$ of integer vectors $m_j = (c_j,d_j)$, the \emph{basic moves}, such that $\pP$ can move by any amount $\lambda m_j$ that takes it to an integral point in $\cB$.  
We refer to a piece with $r$ basic moves as an \emph{$r$-move rider}.
Most important is the \emph{slope}, $\mu_j = d_j/c_j$, a number (or infinity) that is rational for chess pieces, though for counting combinatorial types it need not be rational.

Given a piece $\pP$ with move set $\M = \{m_1,\dots,m_r\}$, place it at $(x_0,y_0)\in\bbZ^2$.  The points $(x,y)$ that $\pP$ attacks are those that satisfy the equation $(x-x_0,y-y_0) = \lambda m_j$ for some $m_j = (c_j,d_j) \in \M$ and real number $\lambda$.  This equation defines a line $l_j$ through $(x_0,y_0)$ of slope $\mu_j=d_j/c_j$, which we call a \emph{move line} of $\pP$.  The $r$ move lines form an arrangement of lines that we call the \emph{move-line arrangement} of $\pP$, written $\cM(\pP)$.  
The move-line arrangement forms $2r$ regions, each bounded by two move lines with consecutive slopes.  

\subsubsection*{Combinatorial type.}
Suppose our $q$ pieces are $\pP_1,\ldots,\pP_q$.  We write $\cM_i = \cM(\pP_i)$ for the arrangement and $l_j^i$ for the move line on piece $\pP_i$ with slope $\mu_j$.  
The \emph{labelled combinatorial type} of a nonattacking configuration of $q$ pieces is the list of the sides of move lines $l_j^i$ occupied by each piece $\pP_k$, for all $(i,j,k)$ such that $i \neq k$.  
We formalize this in two equivalent ways.  
\begin{enumerate}[T1.]
\item First, the $2r$ rays of the $r$ move lines through a piece $\pP_i$ appear in a cyclic order around $\pP_i$.  Each consecutive pair of rays in this order contain between them one region of the move-line arrangement, and each $\pP_k$ (for $k \neq i$) is in one of these regions.  We can number the regions $1,2,\ldots, 2r$ in cyclic order, and use the same numbering for the regions around every other piece (since the move-line arrangement at every piece is a translation of that at $\pP_i$).  With the regions around each piece unambiguously numbered, we can describe a labelled combinatorial type by giving for each ordered pair $(\pP_i,\pP_k)$ the number of the region of $\cM(\pP_i)$ that contains $\pP_k$.  
\item Second, each move line has a preferred orientation given by the direction of its basic move, and therefore it has distinguishable right and left sides.  For each triple $(\pP_i, l_j^i, \pP_k)$ of a piece $\pP_i$, a move line $l_j^i$ on $\pP_i$, and another piece $\pP_k$, the second piece is on the right or left side of $l_j^i$.  The labelled combinatorial type consists of this information.  

If we change the move set $\M$ by reversing a basic move $m_i$, i.e., replacing it by $-m_i$, we do not change the move lines or the relative piece positions but we have reoriented the line $l_i$, interchanging the names of its left and right sides, so we do change the combinatorial type in sense T2.  Thus, the \emph{set} of combinatorial types associated with a move set changes (in a simple way that we call \emph{reorientation}) upon reversing basic moves, but the number of types remains the same, so we can speak of the \emph{number} of combinatorial types associated with $q$ and a set of move lines, independent of the choice of basic move for each line.  
We also say that a set of move lines has a definite set of combinatorial types \emph{up to reorientation}.
\end{enumerate}

The \emph{unlabelled combinatorial type} of a configuration is the labelled type with labels ignored; formally, it is the class of all labelled types obtained from one such type by permuting the labels.  
Another way to make the distinction between labelled and unlabelled types is to characterize the pieces; i.e., combinatorial types of (un)labelled pieces are, respectively, (un)labelled combinatorial types.

\subsubsection*{Isotopy.}
In \cite{QQs1} Chaiken and we defined two kinds of isotopy of nonattacking configurations, each of which produces an equivalence relation on them.  The simpler one is continuous isotopy.  Allow pieces to occupy any real point in the board.  Continuous isotopy means moving the pieces around in any way that keeps the configuration nonattacking throughout.  For discrete isotopy the pieces stay on integral points but we are allowed to refine the grid at will, by increasing $n$ by a large multiplier.  The following lemma is partly explicit and partly implicit in \cite[Section 5]{QQs1}.

\begin{lem}\mylabel{L:top}
Isotopy, discrete isotopy, and having the same labelled combinatorial type produce the same equivalence relation on nonattacking configurations of labelled pieces.
\end{lem}

Consequently, to count combinatorial types it is not necessary to place pieces on integral points of a dilated board; they can be placed anywhere in the polygon $\cB$.

\begin{proof}
The equivalence of isotopy and discrete isotopy is \cite[Theorem 5.4]{QQs1}.  It was stated in \cite{QQs1} for boards with rational vertices but that restriction is unnecessary.  

Obviously, isotopy preserves combinatorial type (labelled, for the proof).  The converse is also true: two nonattacking configurations with the same combinatorial type are isotopic.  This follows from the fact that a combinatorial type corresponds to a region of a hyperplane arrangement in $\bbR^{2q}$, as shown in \cite[Lemma 5.2]{QQs1}.  Thus, isotopy is equivalent to having the same combinatorial type.
\end{proof}

\subsubsection*{Irrationality.}
What we are really doing, in part, is counting the regions of an arrangement of lines of the following form.  There are a set of $r$ slopes and a set of $q$ points in the plane.  Through each of the points there is a line with each of the $r$ slopes.  
In the previous work with Chaiken we assumed rational slopes (counting $\infty$ as an honorary rational slope).  For combinatorial types that is unnecessary, just as it is not necessary to place pieces on board square.  Our arguments for $q=3$ and $r=3$ do not depend on rationality of the slopes.

\subsubsection*{Projective Transformation.}
Suppose $\pP$ is a rider with move set $\M$.  We may convert $\pP$ into another rider $\pP'$ by applying a projective transformation to $\bbR^2$ and $\M$.  Under the transformation, a nonattacking configuration of $q$ copies of $\pP$ (either labelled or unlabelled) becomes a nonattacking configuration of $q$ copies of $\pP'$, which has the same combinatorial type because the regions of $\cM(\pP_i)$ map to the regions of $\cM(\pP_i')$.  Thus the combinatorial types for the two pieces are in natural bijection, although one may have rational slopes and the other not.  We apply this principle later to prove Theorem \ref{T3m}.

There is a subtle aspect to projective transformation.  It can transform any ordered set of three move lines to any other, but it may not transform the first move set to the second move set because it does not necessarily preserve the orientations of the basic moves.  For example, $m_1, m_2, m_3$ may transform to $m_1', m_2', -m_3'$ (respectively) but not to $m_1',m_2',m_3'$.  Thus, projective transformation preserves combinatorial types up to reorientation but not necessarily absolutely.

\section{Three riders and any number of basic moves}

The formula for the number of combinatorial types of nonattacking configuration of three riders is based on counting regions of overlapping move-line arrangements.  We use a classic result of Steiner.

\begin{lem}[Steiner \cite{1826}]\mylabel{L:steiner}
Suppose we have $k$ lines in the plane $\bbR^2$, whose intersection points consist, for each $p\geq2$, of $n_p$ points at which $p$ lines intersect.  Then the number of regions formed by the lines is $1 + k + \sum_{p\geq2} (p-1)n_p$.
\end{lem}

Suppose we have such an arrangement of lines and a board $\cB$.  By taking a sufficiently large dilation of $\cB$ we can ensure that all intersection points are inside the dilation and all regions in the plane do intersect the dilation.  Thus, for counting regions in the dilated board we can ignore the board and pretend it is the whole plane.  That is why the number of combinatorial types of nonattacking configuration is independent of the board.

\begin{thm}\mylabel{T:q3}
On any board, let\/ $\pP$ be a rider with $r$ basic moves.   The number of combinatorial types of\/ $3$ nonattacking unlabelled copies of $\pP$ on dilations of the board equals $\frac13 r(r^2+3r-1)$.  The set of combinatorial types depends only on $r$ and, up to reorientation, is independent of the move set.
\end{thm}

\begin{proof}
We suppose the three pieces are labelled $\pP_1,\pP_2,\pP_3$.  
To make a nonattacking configuration we place $\pP_1$ anywhere in the (open) board, then $\pP_2$ inside a region of the move-line arrangement centered at $\pP_1$, and finally $\pP_3$ inside any region of the combined move-line arrangements of $\pP_1$ and $\pP_2$; that is, of $\cM_{12} = \cM_1 \cup \cM_2$.

According to Lemma \ref{L:top}, the number of labelled combinatorial types is the number of ways to choose the region for $\pP_2$ and then for $\pP_3$.  

When we place $\pP_2$, we put it inside a region of $\cM_1$, which we choose out of $2r$ possible regions.  Due to Lemma \ref{L:top}, any point in that region is equivalent to any other.  

When we place $\pP_3$, we put it in a region of $\cM_{12}$, so we need to calculate the number of regions of this arrangement.  For that, we need to find the intersection points.
First, $\pP_i$ is located on an $r$-fold intersection, as it is on all lines of $\cM_i$ but no other lines (since $\pP_1$ and $\pP_2$ are nonattacking).  Thus, $n_r=2$.
There are no multiple intersections away from the pieces because two lines of the same $\cM_i$ already meet at $\pP_i$.
Line $l_j^1$ intersects every line of $\cM_2$ at a separate point, except for $l_j^2$, to which it is parallel.  Thus, $n_2 = r(r-1)$.

It follows from Lemma \ref{L:steiner} that $\cM_{12}$ has $1 + 2r + [r(r-1) + 2(r-1)] = r^2 + 3r - 1$ regions.

Since $\pP_2$ could have been in any of the $2r$ regions of $\cM_1$, we multiply by $2r$.

Finally, we divide by $3!$ because the three pieces are actually unlabelled.

To complete the proof we show that the combinatorial types of regions of $\cM_{12}$ are independent of where we place $\pP_2$ within a fixed region of $\cM_1$.  We do this by describing all the regions of $\cM_{12}$ and observing that the descriptions depend only on the region of $\cM_1$ occupied by $\pP_2$.  

By rotating the configuration and slopes (as explained under ``Projective transformation'' in Section \ref{sec:back})  we can ensure that the region $C$ of $\cM_1$ in which we will place $\pP_2$ is bounded by the positive horizontal ray out of $\pP_1$.  
By another affine transformation we can arrange the slopes to be $0=\mu_1<\cdots<\mu_{r-1}<\mu_r=\infty$ and $C$ to be the fourth quadrant.  Thus the argument for any choice of $C$ is the same and we can assume a simple form for the move-line arrangements.

The irrelevance of where $\pP_2$ lies in $C$ is obvious from Figure \ref{M12}.
The regions of $\cM_i$ are cones with apex at $\pP_i$ and a region of $\cM_{12}$ is the intersection of two of these cones.  Each cone is either $C_j^i$, bounded by rays with directions $m_{j-1}$ and $m_j$ for $j=2,\dots,r$ or $-m_r$ and $m_1$ for $C_1^i$, or one of the negatives $C_{-j}^i = -C_j^i$.  
The four cones with $j=1$ are special since they are quadrants and the pieces are in two of them: $\pP_2$ in $C_1^1$ with apex $\pP_1$ and $\pP_1$ in $-C_1^2$ with apex $\pP_2$.  
That explains why the exact location of $\pP_2$ is unimportant.
\begin{figure}[htbp]
\begin{center}
\includegraphics[scale=.57]{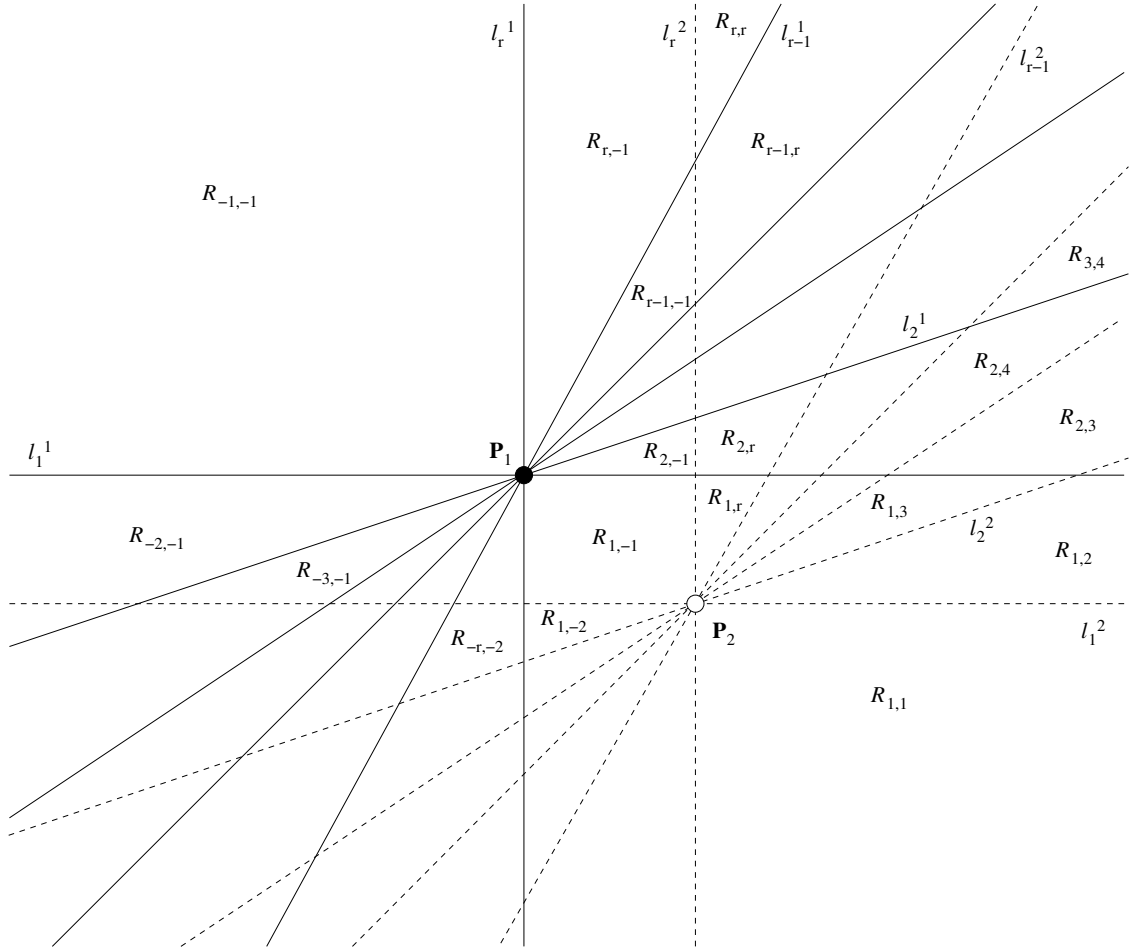}
\caption{The regions of $\cM_{12}$.  It is clear that moving $\pP_2$ left, right, up, or down within $C_1^1$ will not change either the regions formed by intersecting pairs of cones, $R_{jk} = C_j^1 \cap C_k^2$, or the boundary lines of each region.}
\mylabel{M12}
\end{center}
\end{figure}
\end{proof}

\section{Any number of riders having three basic moves}

We have another theorem involving the number three.  We can show that every rider with three basic moves has the same number of combinatorial types of nonattacking configuration.  We do not offer a formula.  We expect it to be complicated---unlike the proof.

\begin{thm}\mylabel{T3m}
On any board, let $\pP$ be any rider with $3$ basic moves.   
The number of combinatorial types of $q$ nonattacking unlabelled copies of $\pP$ on dilations of the board is independent of the move set.
The set of combinatorial types is independent of the move set up to reorientation.
\end{thm}

\begin{proof}
We again rely on Lemma \ref{L:top} to assume pieces can be placed anywhere in the board $\cB$ so long as they do not attack.  Since we need not worry about integral points, a projective transformation does not change the problem.

Every 3-move rider can be transformed by a projective transformation so that its slopes are $0,1,\infty$.  It follows that all 3-move riders are projectively equivalent.  The theorem follows.
\end{proof}

We were led to this theorem by applying Kot\v{e}\v{s}ovec's quasipolynomial formulas to count combinatorial types of $q=4$ of two different 3-move riders: the semiqueen \cite[p.\ 732]{ChMath}, which has the queen's move lines except for one diagonal line, and the trident (his ``bishop + semirook'') \cite[p.\ 730]{ChMath}, which has the queen's move lines without the horizontal.  
We substituted $n = -1$ in both formulas and found the same result:  151 unlabelled combinatorial types.  
\Kot\ has also empirically calculated the counting quasipolynomials for nonattacking configurations of up to $6$ semiqueens and $6$ triangular rooks, pieces that are equivalent to semiqueens on a triangular board.  As Theorem \ref{T3m} predicts, the numbers of unlabelled combinatorial types are the same for both boards even though the complete counting formulas differ.  
See the numbers in Table \ref{Tb:qr} and \Kot's formulas in \cite[Sequences A202654--A202657]{OEIS} for 3 to 6 semiqueens on the square board and \cite[Sequences A193981--A193984]{OEIS} for 3 to 6 triangular rooks.
This is the only comparison we were able to make for $q>3$ and $r\geq3$ of two pieces having the same number of basic moves.

\section{Do fours fail?}\mylabel{sec:fours}

With four or more pieces and basic moves, is there still a single formula in terms of $r$, or in other words, is the number of 
combinatorial types of nonattacking configuration independent of the actual basic moves?  We could not exclude the possibility.  
We studied the simplest case: four 4-move riders.  
When we place the third piece in a region, its exact location influences which combinatorial types are possible for the fourth piece. In Figure~\ref{fig:4x4}(a), with two pieces fixed, the location of the third piece in the shaded region affects the combinatorial types that are possible with a fourth piece. If a fourth piece is placed in a shaded region of Figure~\ref{fig:4x4}(b), the corresponding combinatorial type is not possible in Figure~\ref{fig:4x4}(c). For example, if the fourth piece is at the circle in Figure~\ref{fig:4x4}(b), it is west-northwest of the bottom two pieces and south-southwest of the top piece, which is not a possible placement in Figure~\ref{fig:4x4}(c).

We are led by Theorems \ref{T:q3} and \ref{T3m} and by staring at diagrams like those in Figure \ref{fig:4x4}(b, c) to speculate that the actual combinatorial types of nonattacking configuration may be, up to reorientation, independent of the move set and depend only on $q$ and $r$.  Deciding that looks difficult.

\begin{figure}[hb]
\includegraphics[scale=.25]{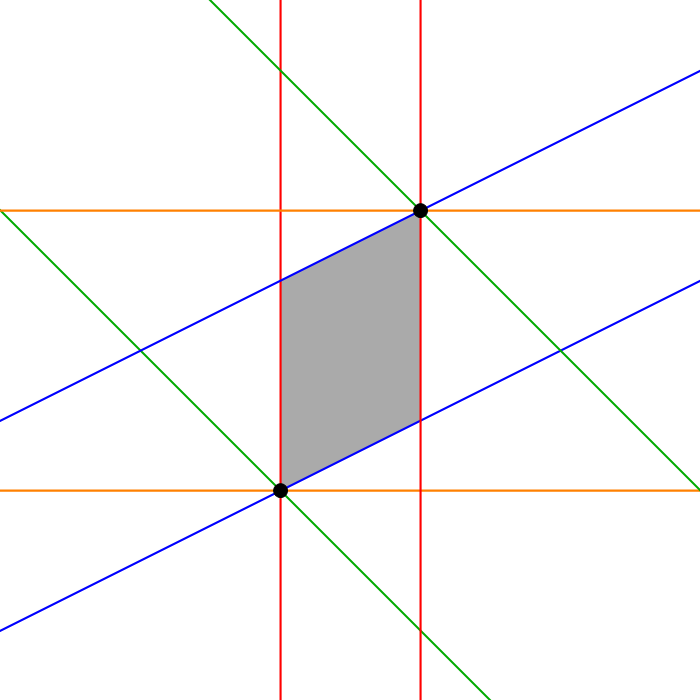}
\\(a)\\[8pt]
\includegraphics[scale=.25]{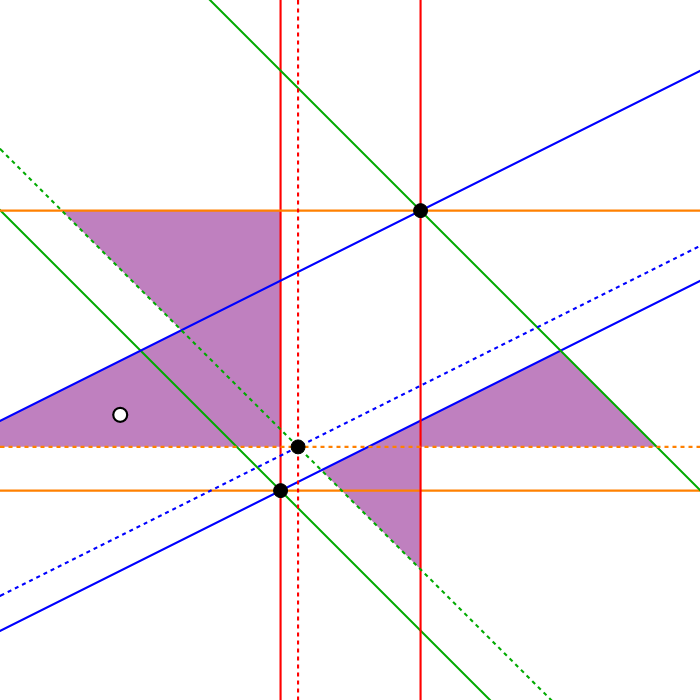} \qquad
\includegraphics[scale=.25]{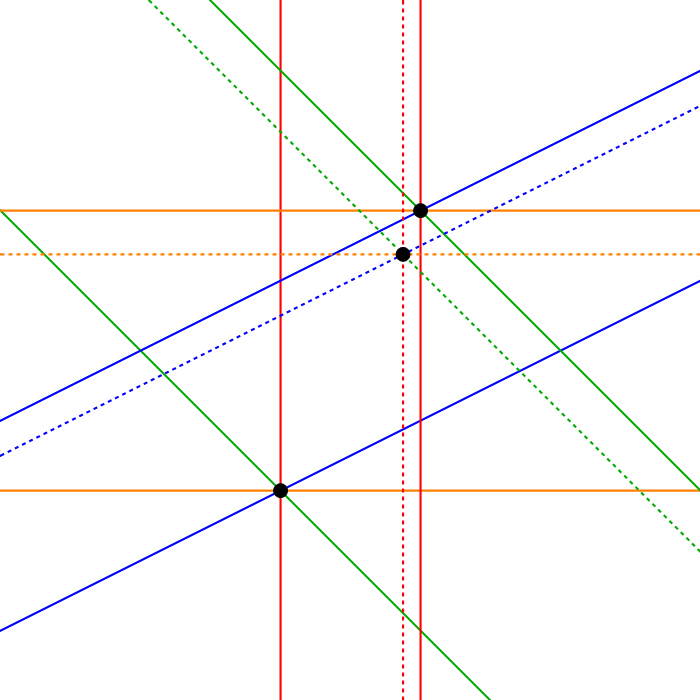}
\\(b) \hspace{2.45in} (c)
\caption{(a)  A nonattacking configuration of two pieces with four move lines.  (b, c) Two ways to place $\pP_3$ that give different possible combinatorial types of location for $\pP_4$.}
\label{fig:4x4}
\end{figure}


\end{document}